\documentclass[a4paper,cleveref,english,nolineno]{gd-lipics-v2}
\usepackage[T1]{fontenc}
\usepackage[utf8]{inputenc}

\usepackage{amsfonts,amsthm,amsmath,amssymb}
\usepackage{cite,microtype,graphicx,hyperref}
\usepackage{subcaption}

\newtheorem{premise}[theorem]{Premise}
\crefname{premise}{Premise}{Premises}

\newcommand{\N}{\mathbb{N}}

\title{String Graph Obstacles of High Girth and of Bounded Degree}

\author{Maria Chudnovsky}{Mathematics Department, Princeton University}{mchudnov@math.princeton.edu}{}{Supported in part by NSF Grants DMS-2348219 and  CCF-2505100.}

\author{David Eppstein}{Computer Science Department, University of California, Irvine}{eppstein@uci.edu}{}{Research supported in part by NSF grant CCF-2212129.}

\author{David Fischer}{Mathematics Department, Princeton University}{df1932@princeton.edu}{}{}

\authorrunning{M. Chudnovsky, D. Eppstein, and D. Fischer}
\Copyright{Maria Chudnovsky, David Eppstein, and David Fischer}



\category{Track 1}

\begin{document}

\EventEditors{Vida Dujmovi\'c and Fabrizio Montecchiani}
\EventNoEds{2}
\EventLongTitle{33rd International Symposium on Graph Drawing and Network Visualization (GD 2025)}
\EventShortTitle{GD 2025}
\EventAcronym{GD}
\EventYear{2025}
\EventDate{September 24--26, 2025}
\EventLocation{Norrk\"{o}ping, Sweden}
\EventLogo{}
\SeriesVolume{357}
\ArticleNo{22}

\maketitle

\begin{abstract}
A string graph is the intersection graph of curves in the plane. Kratochvíl previously showed the existence of infinitely many obstacles: graphs that are not string graphs but for which any edge contraction or vertex deletion produces a string graph. Kratochvíl's obstacles contain arbitrarily large cliques, so they have girth three and unbounded degree. We extend this line of working by studying obstacles among graphs of restricted girth and/or degree. We construct an infinite family of obstacles of girth four; in addition, our construction is $K_{2,3}$-subgraph-free and near-planar (planar plus one edge). Furthermore, we prove that there is a subcubic obstacle of girth three, and that there are no subcubic obstacles of high girth. We characterize the subcubic string graphs as having a matching whose contraction yields a planar graph, and based on this characterization we find a linear-time algorithm for recognizing subcubic string graphs of bounded treewidth. 
\end{abstract}

\section{Introduction}

String graphs, the intersection graphs of curves in the plane~\cite{EhrEveTar-JCT-76}, developed out of the study of the patterns of mutations in DNA sequences~\cite{Ben-PNAS-59}, and of the crossing patterns of wires in circuit designs~\cite{Sin-BSTJ-66}. Beyond the two-dimensional nature of their definitions, string graphs and planar graphs are closely related, although with several important differences. For example, the two classes exhibit similar (but not identical) closure properties.

\begin{lemma} The class of planar graphs is closed under taking minors, and the class of string graphs is closed under taking induced minors~\cite{KraGolKuc-RCA-86}. Explicitly:
    \begin{itemize}
        \item If $G$ is a planar graph, then any graph obtained from $G$ by a sequence of vertex deletions, edge deletions, and edge contractions is also a planar graph.
        \item If $G$ is a string graph, then any graph obtained from $G$ by a sequence of vertex deletions and edge contractions (but not edge deletions) is also a string graph.
    \end{itemize}
\end{lemma}

In general, the class of string graphs is larger than the class of planar graphs, but there is a class of non-string graphs that arises from the class of non-planar graphs. This relationship is formalized in the following lemma, where for a graph $G$, we denote by $\rm{sub}_1(G)$ the graph obtained from $G$ by subdividing each edge once.

\begin{lemma} \label{lem:containments}
    Let $G$ be a graph, and let $H$ denote any graph obtained from $G$ by subdividing each edge at least once, such as $\rm{sub}_1(G)$. Then, the following hold:
    \begin{itemize}
        \item If $G$ is a planar graph, then $H$ is a string graph.
        \item If $G$ is a non-planar graph, then $H$ is a non-string graph~\cite{Sin-BSTJ-66,EhrEveTar-JCT-76}.
    \end{itemize}
\end{lemma}

Famously, the class of planar graphs is characterized by two forbidden minors:

\begin{theorem}[Wagner~\cite{Wag-MA-37}] \label{thm:wag}
    A graph $G$ is non-planar if and only if $G$ contains $K_5$ or $K_{3,3}$ as a minor.
\end{theorem}

As a consequence, planar graphs can be recognized in polynomial time~\cite{RobSey-JCTB-95}.

Due to the relationship between planar graphs and string graphs described by \cref{lem:containments}, a natural question that arises is whether the non-string graphs are exactly those that are obtained by subdividing non-planar graphs. Given \cref{thm:wag}, this idea can be formulated as follows.

\begin{premise} \label{con:gen-duality}
    A graph $G$ is a non-string graph if and only if $G$ contains $\rm{sub}_1(K_5)$ or $\rm{sub}_1(K_{3,3})$ as an induced minor.
\end{premise}

 If \cref{con:gen-duality} were true, it would provide a full structural characterization of the string graphs, allowing for many results from the study of planar graphs to carry into string graphs. However, it is known to be false in general, with string graphs exhibiting much more complex behavior than planar graphs; there are infinitely many \emph{string graph obstacles}~\cite{Kra-JCTB-91a}, i.e. graphs that are not string graphs but for which any proper induced minor is a string graph. Furthermore recognition of string graphs is NP-complete~\cite{Kra-JCTB-91b,SchSedSte-JCSS-03}.

 Given the extensive literature in the study of planar graphs, it is of interest to understand for which subclasses of string graphs the statement of \cref{con:gen-duality} holds; or, more generally, for which subclasses there are only finitely many obstructions. The infinite family of obstacles constructed in \cite{Kra-JCTB-91a} has unbounded clique size, so there is no non-trivial lower bound on the girth of the family, and there is no upper bound on the maximum degree. On the other hand, the non-string graphs $\rm{sub}_1(K_5)$ and $\rm{sub}_1(K_{3,3})$ are both triangle-free and have maximum degree five and three, respectively. Thus, two natural starting points to search for regimes under which \cref{con:gen-duality} holds are the restrictions of string graphs to high girth and to bounded degree. These regimes are the main subject of this paper--both in the context of \cref{con:gen-duality}, and more generally in determining whether these classes contain finitely many or infinitely many string graph obstacles.

In \cref{sec:girth}, we study the regime of graphs with girth at least four, i.e. triangle-free graphs, proving the following.

\begin{theorem}
    There exist infinitely many string graph obstacles of girth four.
\end{theorem}

We note that in addition to being triangle-free, our construction is also $K_{2,3}$-free and near-planar (planar plus one edge---for more on near-planar graphs see \cite{CabMoh-SICOMP-13}). However, it does not have bounded degree.

Next, in \cref{sec:degree} we study the graphs of maximum degree three, the smallest non-trivial upper bound on degree. We prove that even in this regime, \cref{con:gen-duality} does not hold. However, we also show that obstacles in this regime cannot have arbitrarily large girth.

\begin{theorem} \label{thm:intro-degree}
    There exist string graph obstacles of maximum degree three apart from $\rm{sub}_1(K_{3,3})$. All such obstacles have girth smaller than $30$.    
\end{theorem}

To prove \cref{thm:intro-degree}, we obtain a novel characterization of the subcubic string graphs: 

\begin{theorem} \label{thm:intro-recognition}
    Let $G$ be a graph with maximum degree three. Then $G$ is a string graph if and only if there exists a matching $M$ of $G$ such that the graph obtained from $G$ by contracting $M$ is planar.
\end{theorem}

Finally, in \cref{sec:recognition}, we finish by proving an algorithmic corollary of \cref{thm:intro-recognition}:

\begin{theorem}
    The problem of deciding whether a graph with maximum degree three is a string graph is fixed-parameter tractable in treewidth. For any graph class of maximum degree $3$ and bounded treewidth, the problem is decidable in linear time.
\end{theorem}

We note that recognition of string graphs of maximum degree four is an NP-problem \cite{MP21,Kra-JCTB-91b}, whereas the complexity is open for graphs of maximum degree three. It is known that it is NP-complete to check whether a graph has a matching that planarizes the graph when contracted \cite{AH82,AKKV17}, but this is without any restriction on the maximum degree of the input graph.



\section{Preliminaries}

By $\N$ we denote the set of positive integers. We define a \emph{string representation} of a string graph to be a system of curves, called \emph{strings}, associated with the vertices of the graph, that have the given graph as their intersection graph. Rather than considering arbitrary systems of curves in the plane, which may cross each other infinitely many times or even cover nonzero area, it is convenient to consider certain more well-behaved but still fully general string representations. A \emph{Jordan arc} is a curve ambient isotopic to a line segment in the plane. As a degenerate case we also consider a single point to be a Jordan arc. We define a \emph{proper string representation} to be a string representation in which each string is a Jordan arc, each two strings have finitely many points of intersection, at most two strings intersect at any point, and each intersection point is either a crossing of two strings or an endpoint of at least one of the intersecting strings. The proof of the following fact is straightforward and we omit it.

\begin{lemma}
Every string graph has a proper string representation.
\end{lemma}

Each subset of the strings of a proper string representation defines a proper string representation of an induced subgraph. That is not true for the next simplifying assumption to make. A \emph{trimmed string representation} is a proper string representation in which each string is a Jordan arc, and in which it is not possible to replace any string by a proper subset of itself and obtain a string representation of the same graph. That is, each string is minimal among arcs having crossings with the same set of strings. A trimmed string representation may be obtained from a proper string representation by replacing each string by a minimal arc, one at a time.

\begin{observation}
In a trimmed string representation, each string endpoint is an intersection point of two strings. Two strings can intersect either in an endpoint of one or both strings, or in one or more crossing points, but not both.
\end{observation}

\begin{proof}
If two strings had an intersection point at an endpoint, and also another intersection point at a crossing, they would not be minimal for the set of strings that they cross, because the string that ends at an intersection point could be shortened without changing the set of strings that it crosses.
\end{proof}

It is known that, when a graph is formed from another graph by subdividing each edge into a path of two or more edges, the subdivided graph is a string graph if and only if the original graph is planar~\cite{Sin-BSTJ-66}. This idea can be formalized as the following observation:

\begin{figure}[h]
    \centering\includegraphics[scale=0.3]{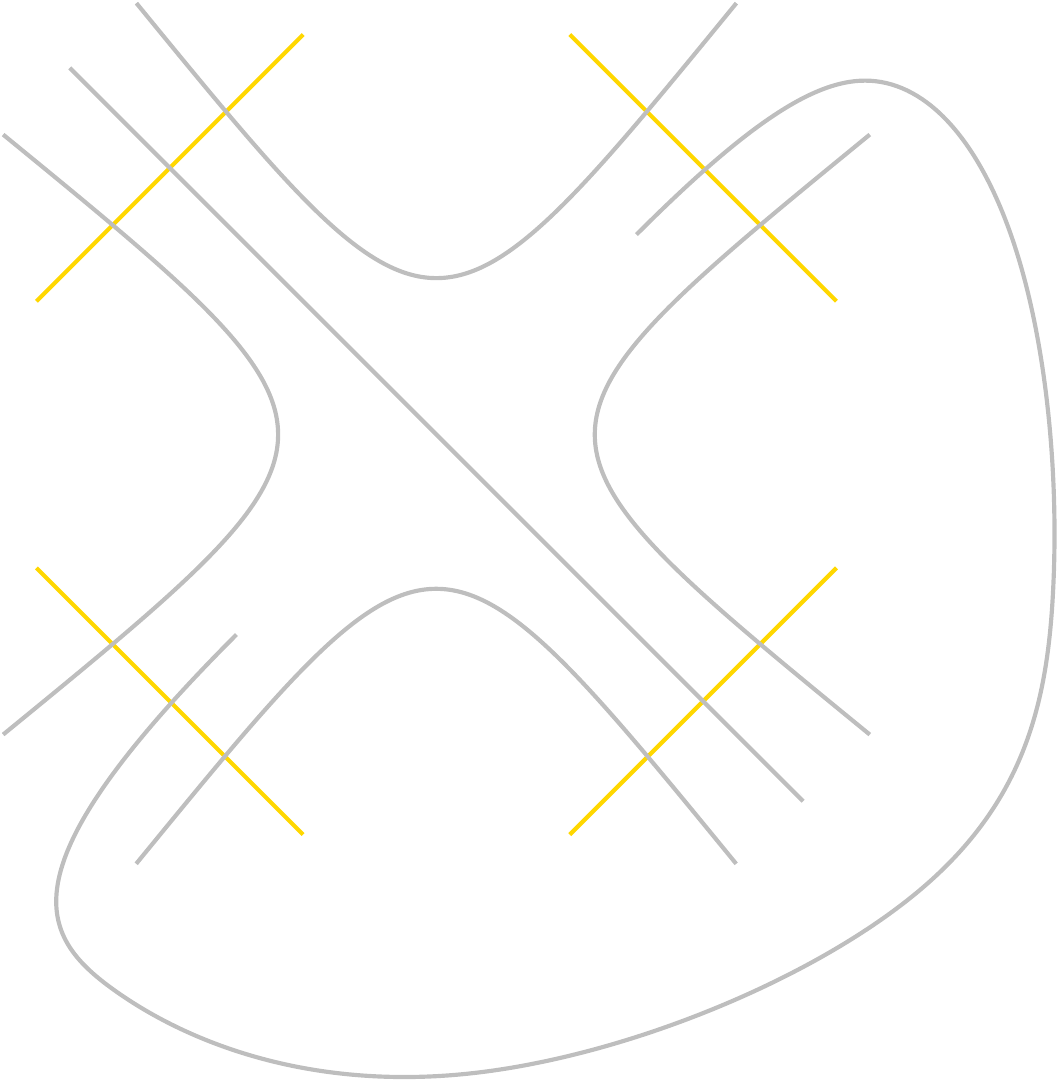}
    \qquad
    \centering\includegraphics[scale=0.3]{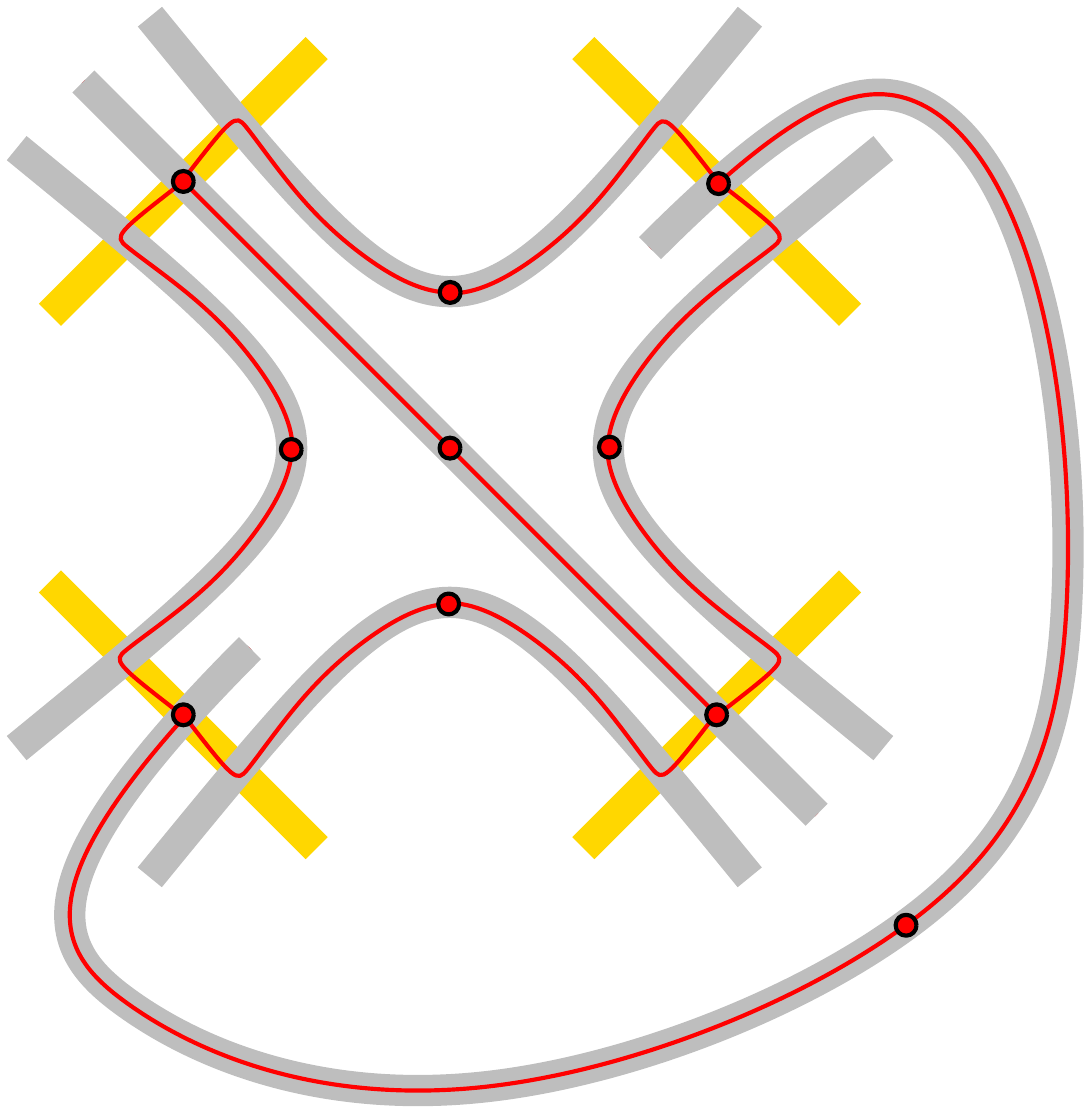}

    \caption{An example of \cref{obs:d2planar}. On the left is a string representation of ${\rm sub}_1(K_4)$, and on the right is a planar representation within (the trimmed part of) a thickening of the string representation.}
\end{figure}

\begin{observation}
\label{obs:d2planar}
If a string graph has the property that every edge has a degree-two endpoint, then it is a planar graph, and any trimmed string representation forms a contact representation by a system of Jordan arcs. If these arcs are thickened to topological disks, then the union of the resulting system of disks contains a planar drawing of the graph in which each vertex is represented by a point within its disk and each edge is represented by a curve within the union of two disks connecting two of these representative points.
\end{observation}

\begin{proof}
No degree-two vertex can have a string with a crossing, because its two endpoints represent its only two adjacencies. If every edge has a degree-two endpoint, there can be no crossings, because each two intersecting strings include one that can have no crossings.
\end{proof}

Here the planarity of the resulting graph requires that we are considering contact graphs of curves for which at most two can intersect in any point. Contact systems of curves without a restriction to pairwise intersections may be nonplanar~\cite{Hli-GD-95}.

\section{Triangle-free Obstacles} \label{sec:girth}

\begin{figure}[h]
\centering\includegraphics[height=4in]{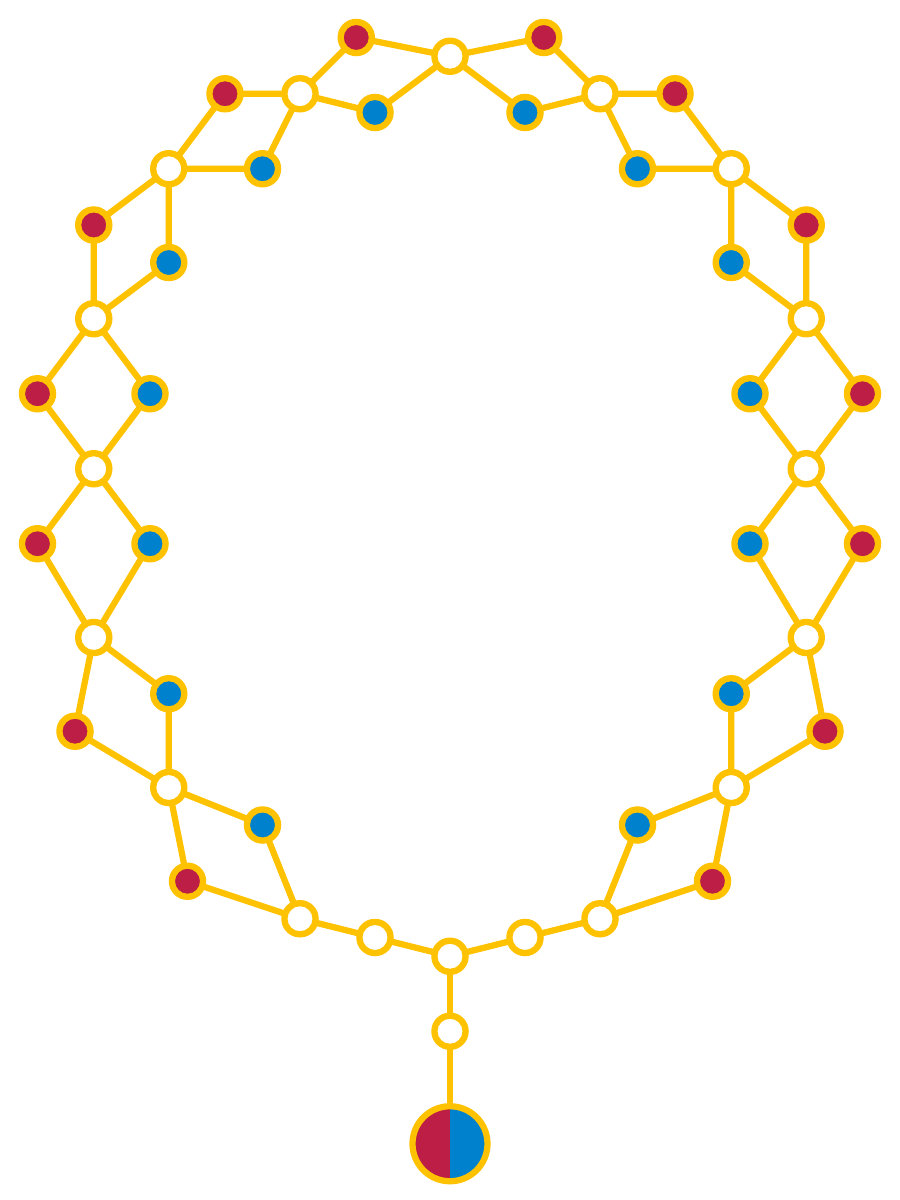}
\caption{The necklace graph $N_{14}$.}
\label{fig:necklace}
\end{figure}

Our family of triangle-free obstacles is built from a family of vertex-colored planar graphs, an example of which is depicted in \cref{fig:necklace}.
This graph has the overall form of a cycle, in which components in the form of two different types of subgraph are connected by pairs of \emph{terminal vertices}, colored white. Each component has two terminal vertices, each one shared with one of its two neighboring components in the cycle.
The two types of components are:
\begin{itemize}
\item Any number of \emph{links}, four-vertex cycles consisting of two non-adjacent white terminals, a red vertex, and a blue vertex.
\item The \emph{pendant}, a tree formed as the 1-subdivision of the claw $K_{1,3}$. All but one vertex are white; the remaining vertex, a leaf of the tree, is colored both red and blue, and we call it the \emph{jewel}. The two terminals are the other two leaves of the tree.
\end{itemize}
For each $i \in \N$, we define the necklace $N_i$ to be a necklace whose cycle of components consists of  $i$ links and one pendant. For instance the necklace in the figure can be denoted in this way as $N_{14}$. We observe that all planar drawings of this graph are topologically equivalent if neither the vertex colors nor the distinction between bounded and unbounded faces is considered. Each link bounds a quadrilateral face. There are two more faces, one polygonal face with $4+2i$ sides not containing the jewel, and a face with $8+2i$ sides (some of them repeated) containing the jewel. By convention we call these last two faces the \emph{inner face} and the \emph{outer face}, respectively, matching the way that they are drawn in the figure. Alternative planar embeddings may rearrange the colors of some of the vertices on these faces but not their topological structure.

We are now ready to describe our family of triangle-free obstacles. They are the graphs $\hat N_i$, obtained from a necklace $N_i$ by adding two more vertices: a red apex and a blue apex. The red apex is adjacent to all the red vertices of $N_i$, including the jewel. The blue apex is adjacent to all the blue vertices of $N_i$, including the jewel. These are the only additional adjacencies.

\begin{observation}
\label{obs:near-planar}
For each $i \in \N$, $\hat N_i$ is triangle-free, $K_{2,3}$-subgraph-free, and near-planar.
\end{observation}

\begin{proof}
Each component of $N_i$ is triangle-free, and the inner and outer face are too long to have a triangle, so $N_i$ is triangle-free. No two red vertices are adjacent, so there is no triangle involving the red apex, and symmetrically no two blue vertices are adjacent, so there is no triangle involving the blue apex.
At most two blue vertices of $N_i$ share a common white neighbor, so no $K_{2,3}$ subgraph can include the blue apex, and symmetrically with the red apex, so any $K_{2,3}$ subgraph of $\hat N_i$ can only exist within $N_i$ itself. However, the only four-cycles in $N_i$ are the links, which do not form $K_{2,3}$ subgraphs.

When the necklace is drawn as in the figure, with all blue vertices on the inner face and all red vertices on the outer face, the blue apex can be placed in the inner face and connected planarly to all blue vertices except the jewel. The red apex can be placed in the outer face and connected planarly to all red vertices including the jewel. Only one edge, from the blue apex to the jewel, is missing.
\end{proof}

\begin{figure}[t]
\centering\includegraphics[scale=0.3]{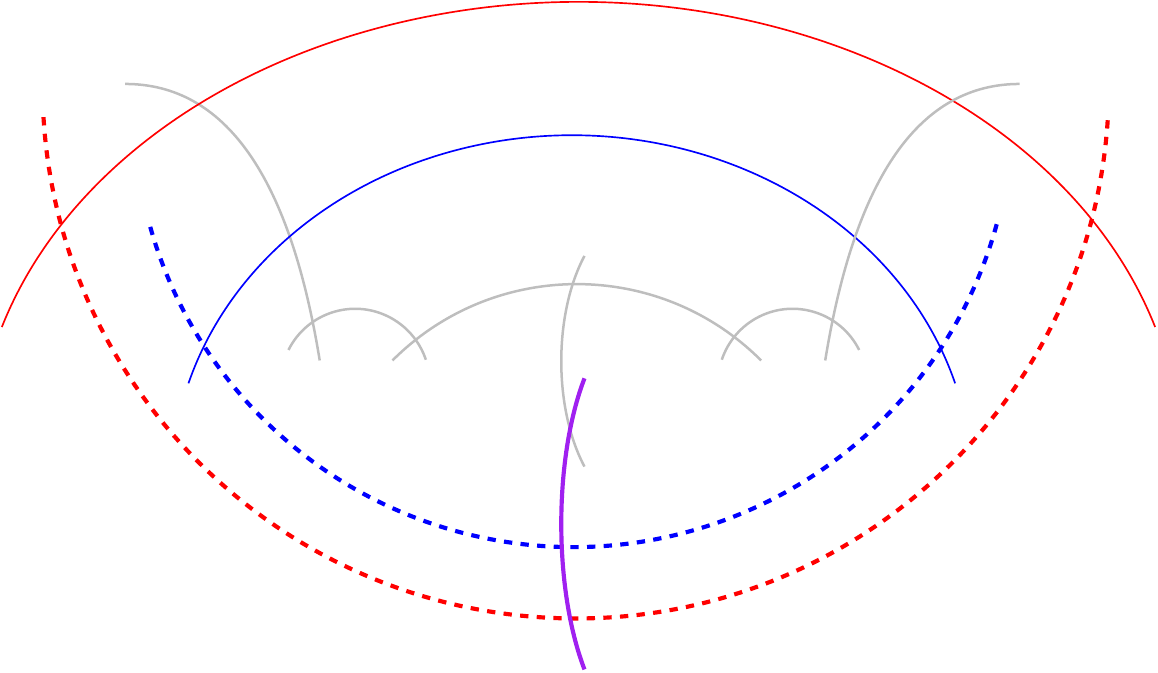}
\qquad
\centering\includegraphics[scale=0.3]{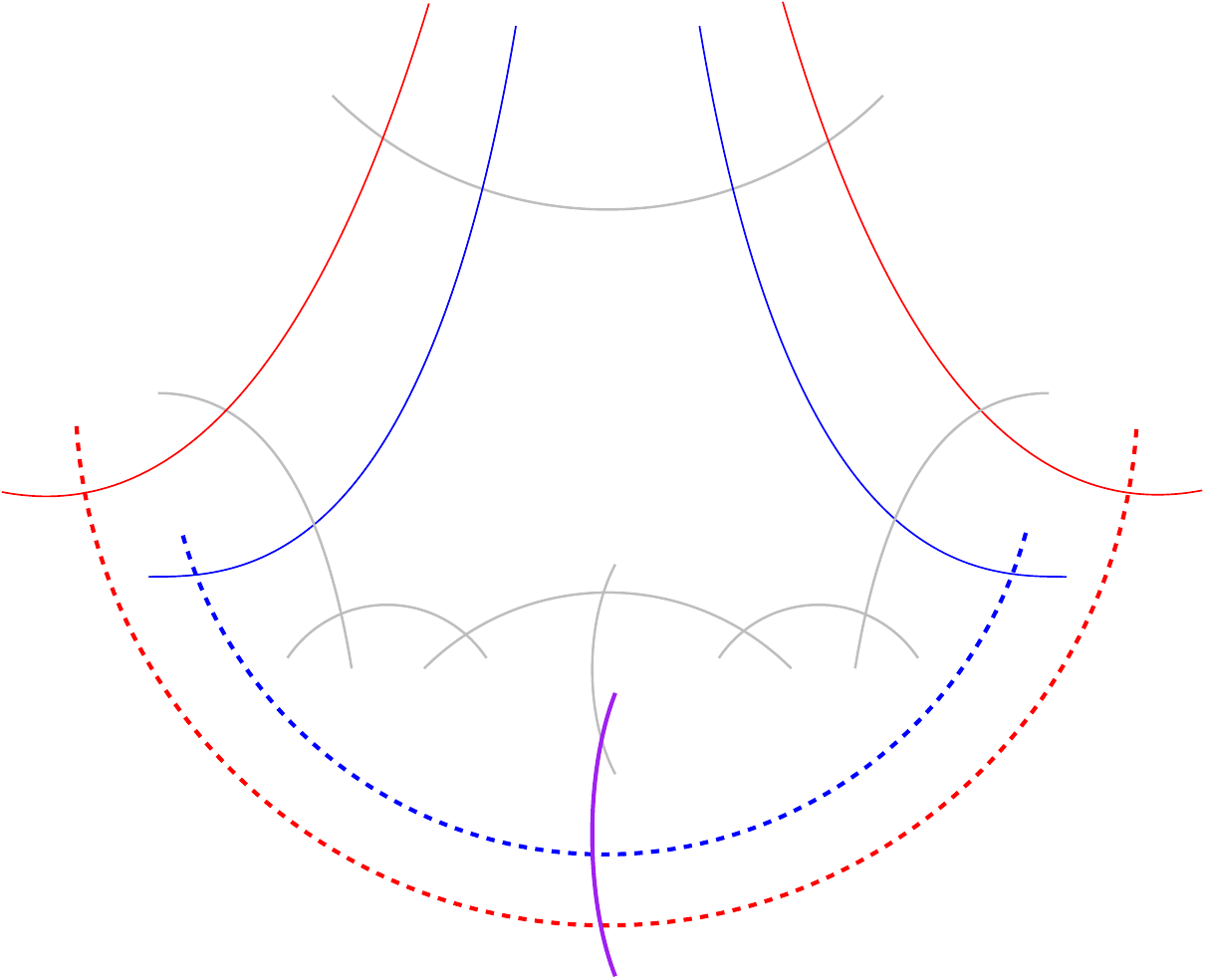}

\caption{String representations of $\hat N_1$ (left) and $\hat N_2$ (right). The light gray strings represent white vertices, and the solid red and blue strings represent red and blue vertices, respectively. The dashed red and blue strings represent the red and blue apexes, respectively, and the purple string represents the jewel.}
\label{fig:N-string}
\end{figure}

\begin{lemma}
\label{lem:nonstring}
For $i \in \N$, $\hat N_i$ is a string graph if and only if $i \in \{1,2\}$.
\end{lemma}

\begin{proof}
String representations of $\hat N_1$ and $\hat N_2$ are shown in \cref{fig:N-string}. 

We now assume that $i \geq 3$. Suppose that $S$ is a string representation of $\hat N_i$. The induced subgraph $N_i$ obeys \cref{obs:d2planar}; obtain a trimmed representation of $N_i$ by removing the strings for the two apexes and trimming,
and let $D$ be a planar drawing of $N_i$, within a thickening of this trimmed representation, following \cref{obs:d2planar}.

Observe that $N_i$ contains a unique cycle $C_R$ comprised only of red and white vertices, and similarly $N_i$ contains a unique cycle $C_B$ comprised only of blue and white vertices. Then any cycle in $D$ separates the plane into two sides, restricting any string in $S$ that does not have a neighbor in the cycle to only one of these two sides. Thus, in $D$, $C_R$ separates the plane into two faces, which we denote by $F_{C_R}^+$ and $F_{C_R}^-$. Similarly, $C_B$ separates the plane into two faces, which we denote by $F_{C_B}^+$ and $F_{C_B}^-$. Since the jewel does not have a neighbor in $C_R$, its string in $S$ cannot pass through both $F_{C_R}^+$ and $F_{C_R}^-$; without loss of generality, we assume that it only passes through $F_{C_R}^+$. Similarly, we assume that the string of the jewel only passes through $F_{C_B}^+$ and does not pass through $F_{C_B}^-$. Thus, the jewel string lies entirely in $F_{C_R}^+ \cap F_{C_B}^+$.

We observe that the red apex only leaves $F_{C_R}^+ \cap F_{C_B}^+$ in order to enter sections of the plane that are bound by $4$-cycles of $D$, and each such cycle $C$ consists of one red vertex, one blue vertex, and two white vertices. In particular, as the red vertex is the only vertex in $C$ that is adjacent to the red apex, the red apex must enter and exit the section bound by $C$ through the red vertex. It is easy to see that if this is the case, then $S$ can be modified so that the red apex string does not enter the section bound by $C$, by extending the red vertex's string to intersect the red apex string within $F_{C_R}^+ \cap F_{C_B}^+$. Thus, without loss of generality, we may assume that the red apex string lies entirely in $F_{C_R}^+ \cap F_{C_B}^+$, and symmetrically that the blue apex string also lies entirely in $F_{C_R}^+ \cap F_{C_B}^+$. From here, we deduce that there exists a string representation of $\hat N_i$ if and only if there exists a graph $H_i$ and a coloring $c: V(H_i) \to \{R,B\}$ such that
\begin{enumerate}
    \item $V(H_i) = \{v_1^1,v_1^2,v_2^1,v_2^2,\ldots, v_n^1,v_n^2\}$, and
    \item $v_1^1v_1^2\ldots v_n^1v_n^2v_1^1$ is a cycle (which we will denote by $C_i$), and
    \item $c(v_j^1) \neq c(v_j^2)$ for all $1 \leq j \leq n$, and
    \item $c(x) = c(y)$ for all $xy \in E(H_i) \setminus E(C_i)$, and
    \item \label{lem-con:connectivity} $H_i[R_i] \setminus E(C_i)$ and $H_i[B_i] \setminus E(C_i)$ are both connected, where $R_i = \{v \in V(H_i) : c(v) = R\}$ $B_i = \{v \in V(H_i) : c(v) = B\}$, and
    \item \label{lem-con:planar} there exists a planar drawing of $H_i$ such that all edges lie entirely in the exterior of $C_i$.
\end{enumerate}

The intuition behind the equivalence between these two problems is that, if a string representation of $\hat N_i$ exists, then the string representation gives rise to an $H_i$ as above, with the apex strings defining the non-cycle edges of $H_i$. Conversely, an $H_i$ as described above gives rise to a string representation of $\hat N_i$, with the non-cycle edges of each color defining that color's respective apex string. This relationship is shown in \cref{fig:N4}.

Regardless of the choice of $H_i$ and $c$, there exist vertices $r_1,r_2,b_1,b_2 \in V(H_i)$ with $c(r_1) = c(r_2) = R$ and $c(b_1) = c(b_2) = B$ such that $r_1r_2, b_1b_2 \in E(H_i)$ and the cyclic order of these vertices on $C_i$ is $r_1,b_1,r_2,b_2$. It follows that in any planar drawing of $H_i$, it is impossible for both $r_1r_2$ and $b_1b_2$ to lie entirely in the exterior of $C_i$ without crossing. Thus, $\hat N_i$ is not a string graph when $i \geq 3$.
\end{proof}

\begin{figure}[t]
    \centering\includegraphics[scale=0.43]{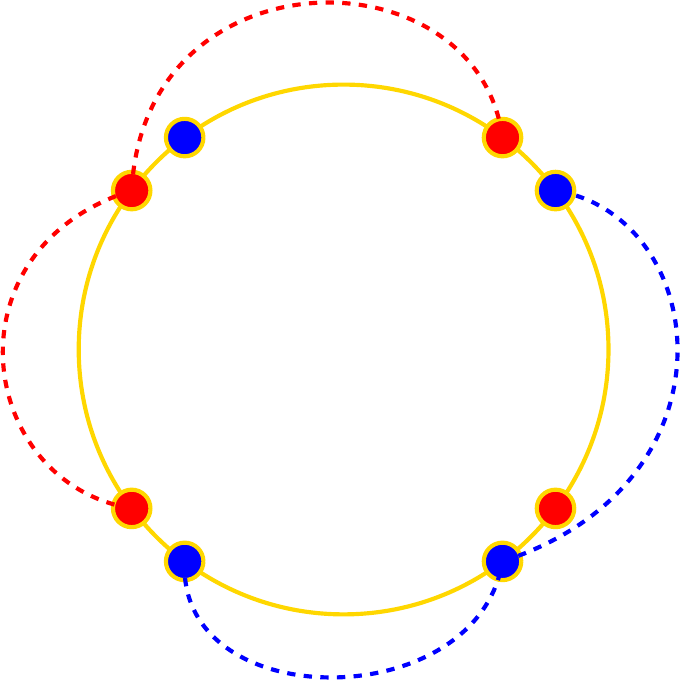}
    \qquad
    \centering\includegraphics[scale=0.3]{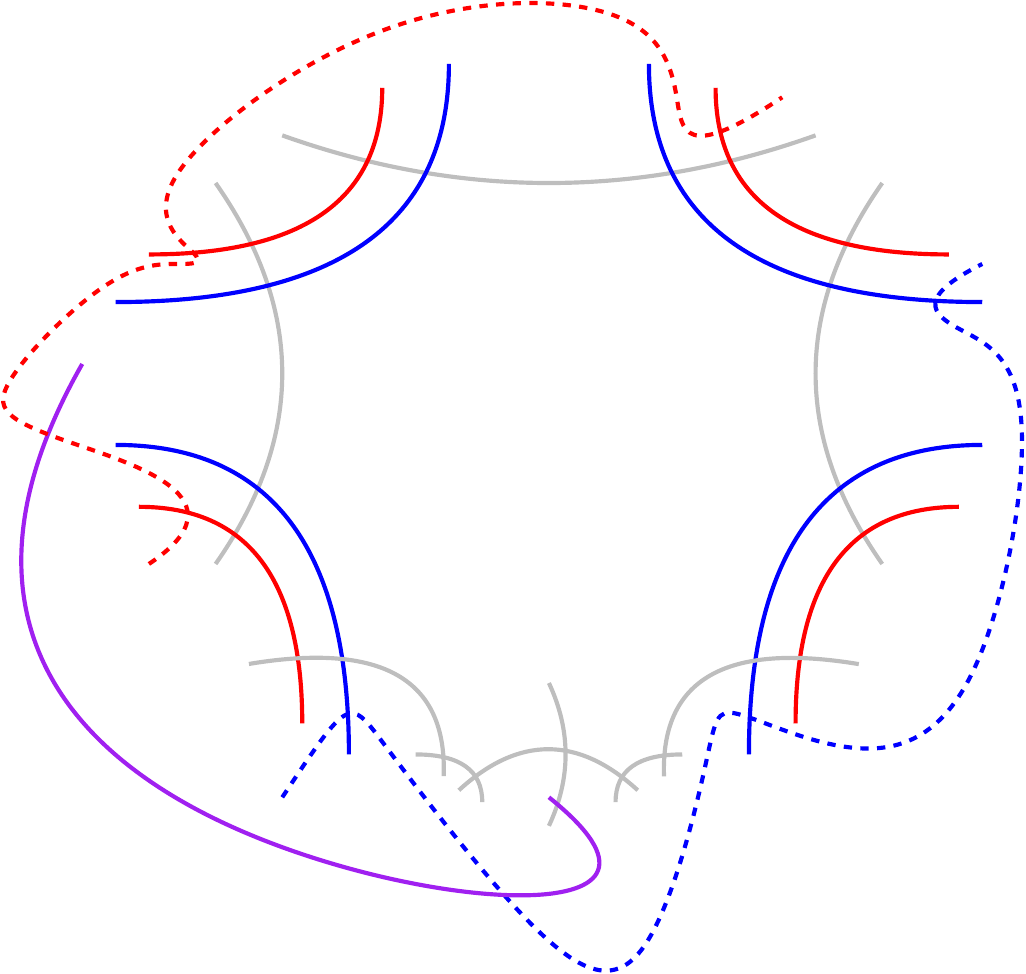}

    \caption{An example demonstrating the idea of the reduction in the proof of \cref{lem:nonstring}. On the left is an attempt at creating a satisfying $H_4$ as described in the proof; the only condition violated is Condition \ref{lem-con:connectivity}, that each of the monochromatic induced subgraphs (i.e. the subgraph on all red vertices and edges, and the one on all blue vertices and edges) is connected, as each is the disjoint union of a path on $3$ vertices and an isolated vertex. However, any attempt to add edges to satisfy Condition \ref{lem-con:connectivity} causes the graph to violate Condition \ref{lem-con:planar}. On the right is the corresponding attempt to draw a string representation of $\hat N_4$. Each apex string intersects (apart from the jewel) three of the four strings that it needs to in order to achieve a string representation of $\hat N_4$, and these three strings correspond to the path on three vertices in the monochromatic subgraph of the left graph, while the string not intersected by the apex represents the isolated vertex in the monochromatic subgraph of the left graph. }
    \label{fig:N4}
\end{figure}

\begin{lemma}
\label{lem:minimal}
For each $i \in \N$, every proper induced minor of $\hat N_i$ is a string graph.
\end{lemma}

\begin{proof}
We need only consider induced minors that perform a single vertex deletion or a single edge contraction in $\hat N_i$. Let $e$ be the edge connecting the jewel to the blue apex; we will start with the planar drawing of $\hat N_i-e$ obtained in \cref{obs:near-planar}, and its string representation obtained from this drawing by tracing a string around each vertex and its incident half-edges. We will show that in each possible vertex deletion $\hat N_i-v$ or edge contraction $\hat N_i/f$, it is possible to modify the string representation of $\hat N_i-e$ to obtain a string representation of $\hat N_i-v$ or $\hat N_i/f$.

\smallskip\noindent
For vertex deletions $\hat N_i-v$, we have the following cases:
\begin{itemize}
\item If $v$ is a white vertex belonging to both the inner and outer face of $N_i$, then deleting it and its string produces a string representation of $\hat N_i-v$ in which the inner and outer faces are no longer separated from each other. We can extend the string for the jewel to cross through the string for the red apex and through this gap, from the outer face to the inner face, where it can cross the string for the blue apex without being obstructed by any other string.
\item If $v$ is a red vertex, then deleting it and its string produces a string representation of $\hat N_i-v$ that breaks the red--white cycle which blocked the blue apex from passing through both the inner and outer faces. The string for the blue apex can be extended to cross through the blue vertex in the same link as $v$, into the outer face of $N_i-v$. It cannot cross the string for the red apex, but the string for the jewel can also be extended to cross the string for the red apex and meet the extended string for the blue vertex, adding the missing edge $e$ to the string representation of $\hat N_i-v$.
\item If $v$ is a blue vertex of $N_i$, then deleting it is symmetric to deleting a red vertex under a symmetry of $\hat N_i$ that swaps the two colors red and blue.
\item If $v$ is the blue apex or the jewel, then deleting it obviates the need to represent the missing edge $e$. The case where $v$ is the red apex is symmetric.
\item The only remaining vertex that could be deleted is the white vertex neighboring the jewel. If this vertex is deleted, the string for the jewel can be placed within any link. The strings for each apex can be extended through the vertex of the link of the same color, into the link, where they can separately intersect the string for the jewel.
\end{itemize}

\smallskip\noindent
For edge contractions  $\hat N_i/f$, we have the following cases:
\begin{itemize}
\item If $f$ is a blue--white edge of a link, contracting it produces a contracted vertex that is blue (incident to the blue apex) and that belongs to both the inner and outer face. The string for the blue apex can cross through the string for this contracted vertex, into the outer face, where it can be met by an extension of the string for the jewel, much like the case of deleting a red vertex. Contracting a red--white edge of a link is symmetric.
\item If $f$ is an edge connecting the blue apex to a blue vertex, contracting it produces a graph in which the blue apex has a white neighbor on both the inner and outer face. The string for the blue apex can extend through the string for this neighbor, into the outer face, where it can be met by an extension of the string for the jewel. Contracting an edge connecting the red apex to a red vertex is symmetric.
\item If $f$ is an edge on the four-edge path connecting the two terminals of the pendant, then contracting it produces a graph in which the degree-3 vertex of the pendant is adjacent to one terminal. The string for this degree-3 vertex can cross through the neighboring terminal, into its adjacent link. This allows the jewel and its neighbor to be placed within this link, instead of in the inner or outer face. The strings for each apex can cross through the strings for the red and blue vertices of this link, allowing the apex strings to reach the string for the jewel.
\item If $f$ is one of the two remaining edges of the pendant, then contracting it makes the jewel adjacent to the white degree-3 vertex, which belongs to both the inner and outer face. The string for the jewel can cross through the string for this white neighbor, into the inner face, where it can be intersected by the string for the blue apex.
\item If $f$ is the edge connecting the red apex to the jewel, then after it is contracted and the two strings for the apex and jewel are combined into a single string, this combined string still needs to intersect the string for the blue apex. It can do so by crossing the string for any red vertex into any link, where it can be met by an extension of the blue apex string that crosses the string for the blue vertex of the same link. The case where $f$ connects the blue apex to the jewel is symmetric. 
\end{itemize}

\smallskip\noindent
Each case produces a valid string representation, from which the result follows.
\end{proof}

\begin{theorem}
There are infinitely many triangle-free, $K_{2,3}$-subgraph-free, near-planar graphs that are minimal as non-string graphs in the induced minor order.
\end{theorem}

\begin{proof}
The graphs $\hat N_i$ for $i\ge 3$ are non-isomorphic (they have different numbers of vertices), triangle-free, $K_{2,3}$-free, and near-planar by \cref{obs:near-planar},
non-string graphs by \cref{lem:nonstring}, and minimal with this property for the induced minor order by \cref{lem:minimal}.
\end{proof}

\section{Subcubic Obstacles} \label{sec:degree}

\begin{figure}[t]
\centering\includegraphics[width=0.4\textwidth]{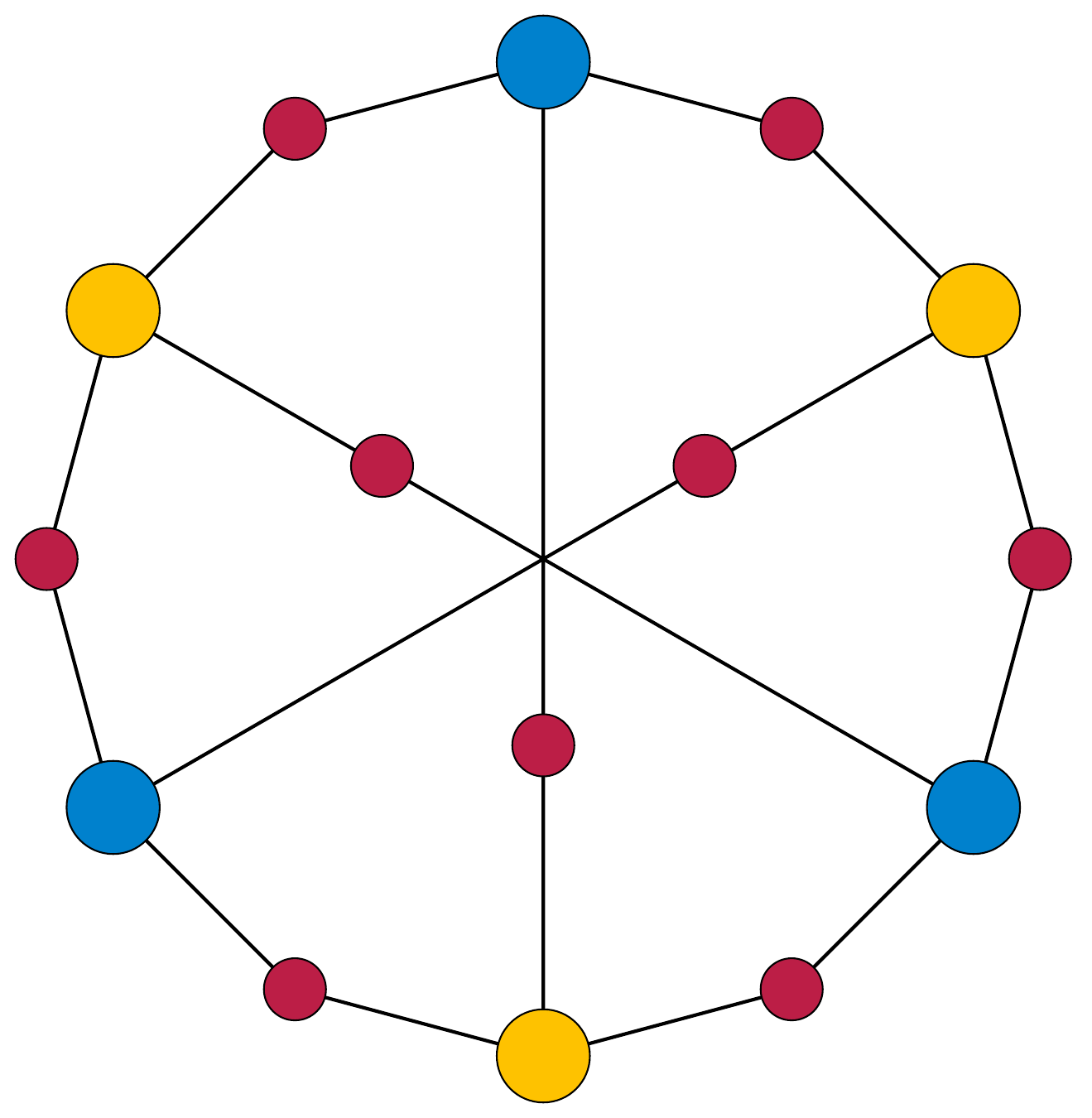}
\caption{The 1-subdivision of $K_{3,3}$, the only previously known subcubic obstacle to string graphs.}
\label{fig:1sub-k33}
\end{figure}

A graph $G$ is \emph{subcubic} if its maximum degree is at most three, and a \emph{cubic matching of $G$} is a matching $M$ using only cubic edges (edges that have two degree-3 endpoints). The only previously known subcubic obstacle to string graphs was the 1-subdivision of $K_{3,3}$ (\cref{fig:1sub-k33}), which has girth eight. In this section we present new subcubic obstacles to string graphs including one that contains a triangle. Additionally, we limit the search for other subcubic obstacles by proving that they do not have arbitrarily large girth.

We do not know whether the string graphs have a finite or infinite number of subcubic obstacles. It is known that there exist infinite antichains of bounded-degree graphs in the induced minor order~\cite{MatNesTho-CMUC-88}, but even this much appears not to have been known for maximum degree three. In \cref{sec:d3-antichain} we describe an infinite antichain of cubic toroidal graphs, and relate them to string graphs, however without determining whether they can be used to construct additional subcubic obstacles.

\begin{figure}[t]
\centering\includegraphics[scale=0.2]{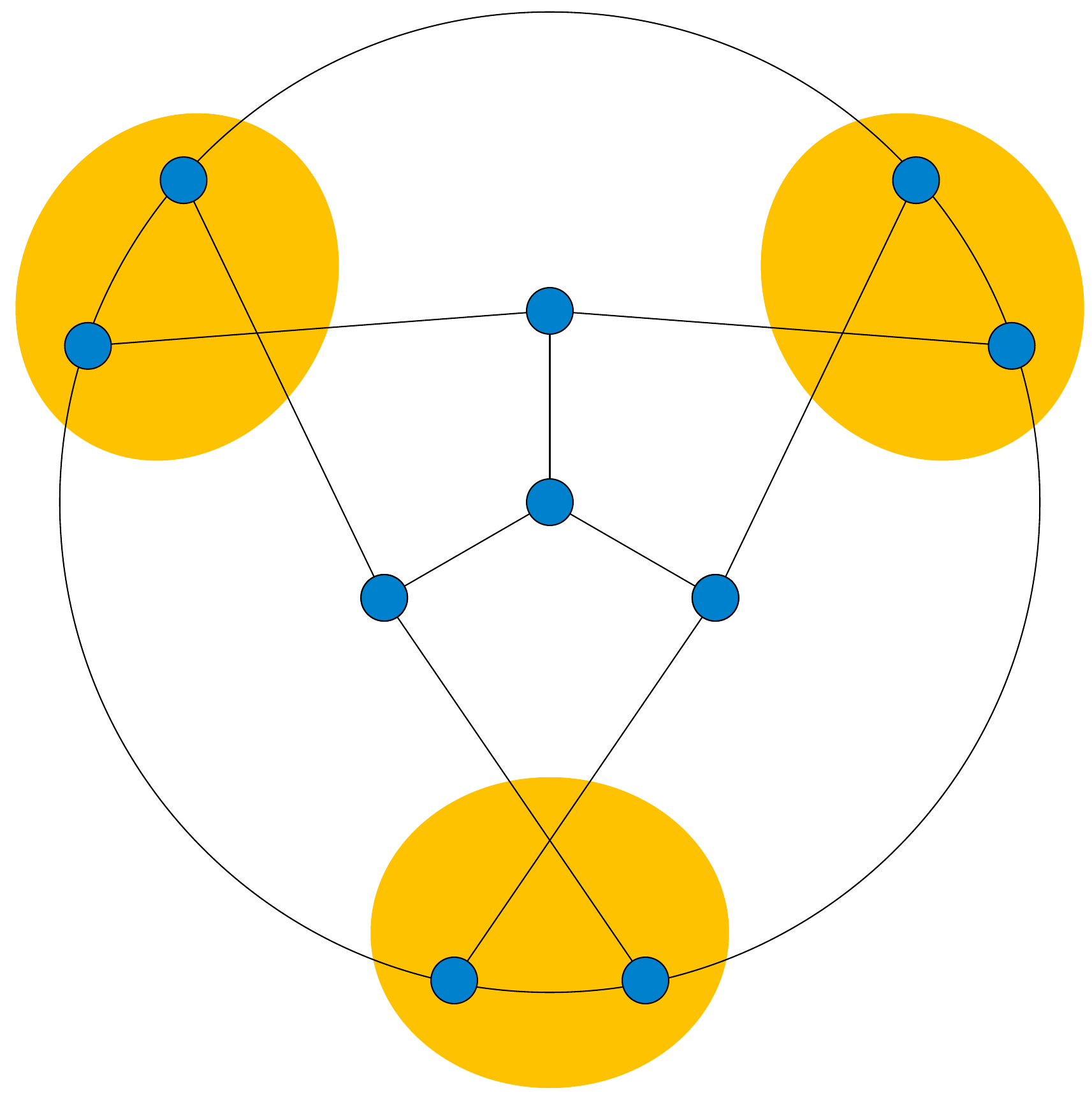}
\qquad
\centering\includegraphics[scale=0.2]{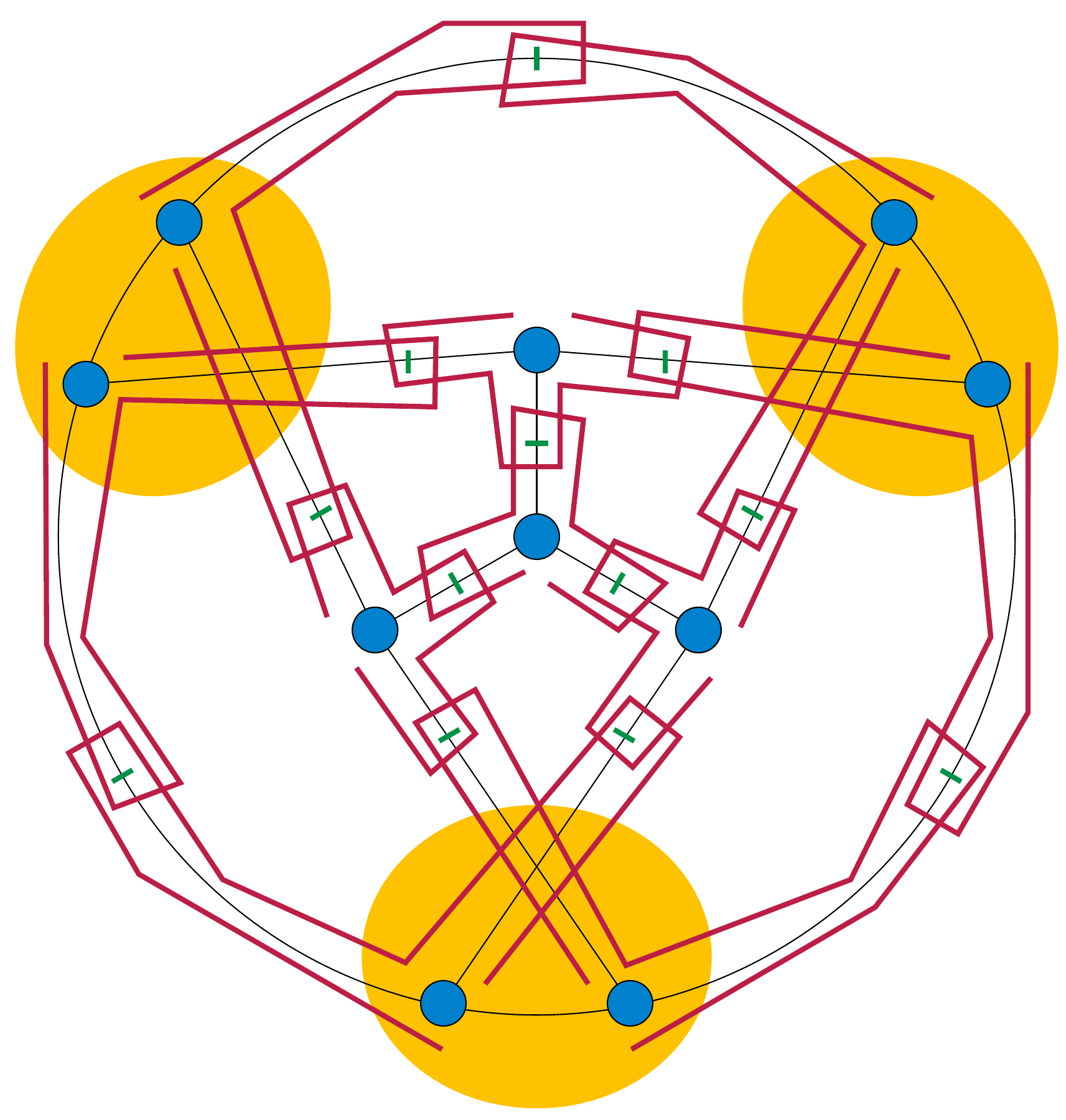}

\caption{Left: A three-edge matching (shaded ovals) in the Petersen graph whose contraction produces a planar graph (all crossings are inside ovals). Right: A string graph representation of the Petersen graph constructed following the proof of \cref{thm:subcubic-string}. The green marks subdivide each unmatched edge into half-edges, which the red strings trace around.}
\label{fig:petersen}
\end{figure}

We base many of our results in this section on the following characterization of subcubic string graphs in terms of contractions of matchings.

\begin{theorem}
\label{thm:subcubic-string}
If a graph $G$ has a matching $M$ such that the contraction of the matched edges, $G / M$, is planar, then $G$ is a string graph.
Every subcubic string graph has such a matching that is cubic.
\end{theorem}

\begin{proof}
For the first claim of the theorem, suppose that $G$ has a matching $M$ for which $G/M$ is planar; we show that $G$ is a string graph. Find a planar drawing of $G/M$, and expand each vertex in the drawing, coming from a contracted edge in $M$, from a point into a topological disk, disjoint from the remaining components of the drawing. \cref{fig:petersen} depicts an example, with these expanded disks shown in yellow. Let $D$ be the resulting vertex-expanded drawing. Place a point arbitrarily on each edge of $D$, disjoint from the endpoints of the edge, separating the edge into two topological half-edges. Represent $G$ as a string graph in which each unmatched vertex is represented by a string that traces around the vertex and its incident half-edges in $D$.
For each vertex $v$ in a matched edge $vw$, represent $v$ as a string which stays within the expanded disk representing the contracted vertex for $vw$, crossing the string for $w$, and extends beyond this expanded disk only to trace around each half-edge connecting this contracted vertex to a neighbor of $v$ in $G$. The resulting system of strings intersect within the expanded disks for each contracted edge, and near the subdivision point along each uncontracted edge representing an adjacency in $G/M$; there are no other crossings. Thus, these strings form a string representation of $G$. 

For the second claim of the theorem, suppose that $G$ is a subcubic string graph. We must show that there exists a cubic matching $M$ such that $G/M$ is planar. Consider any trimmed representation of $G$. Each string in this representation, representing a vertex $v$, must touch strings for two distinct neighbors of $v$ at its two endpoints; it can intersect at most one other string than these two neighbors, a third neighbor $w$ of $v$, at proper crossings. If the string for $v$ crosses the string for $w$ in this way, for the same reason, $w$ can have no other crossings than with $v$. Thus, the pairing of vertices of $G$ whose strings cross each other is a matching $M$. When the two strings for each matched pair are replaced by their union, a connected subset of the plane, the result is a contact representation of $G/M$, by sets (curves and unions of two curves) for which at most two sets intersect at any point, which is therefore planar as in \cref{obs:d2planar}. Furthermore, two strings can only intersect at a proper crossing if both vertices have degree three. Thus, every edge of $M$ is adjacent to two vertices of degree three, so $M$ is a cubic matching.
\end{proof}

\subsection{New Subcubic Obstacles}

As an example, this characterization can be used to show that the 1-subdivision of $K_{3,3}$ (\cref{fig:1sub-k33}) is a minimal non-string graph.
It is not a string graph by \cref{obs:d2planar}. If any vertex is deleted, the remaining subgraph becomes planar, because it forms a subdivision of a vertex deletion or edge deletion of $K_{3,3}$. And if any of its edges
is contracted, replacing a two-edge subdivided path $uvw$ by the single edge $uw$,
then the matching consisting only of edge $uw$ meets the conditions of \cref{thm:subcubic-string}: contracting it produces a subdivision
of the planar wheel graph obtained by contracting one edge of $K_{3,3}$. Thus, all induced minors of the 1-subdivision of $K_{3,3}$ are string graphs.

We use this characterization to show the existence of several novel subcubic obstacles, depicted in \cref{fig:subcubic-obstacles}. The first is $KT_3$, which is constructed from the 1-subdivision of $K_{3,3}$ by adding an edge between two subdivision vertices at distance two from each other.

\begin{figure}[t]
    \centering\includegraphics[width=0.4\linewidth]{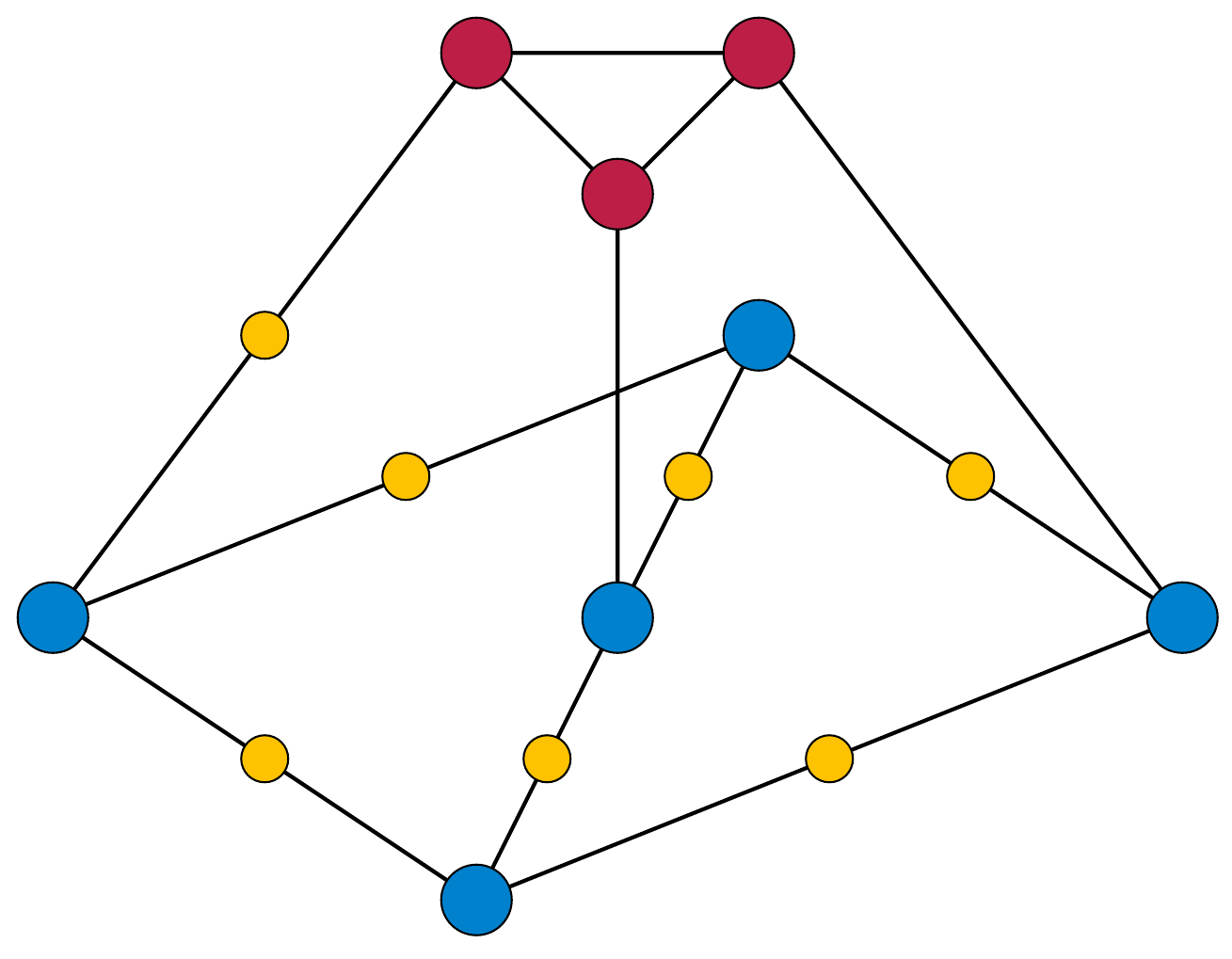}
    \qquad
    \centering\includegraphics[width=0.4\linewidth]{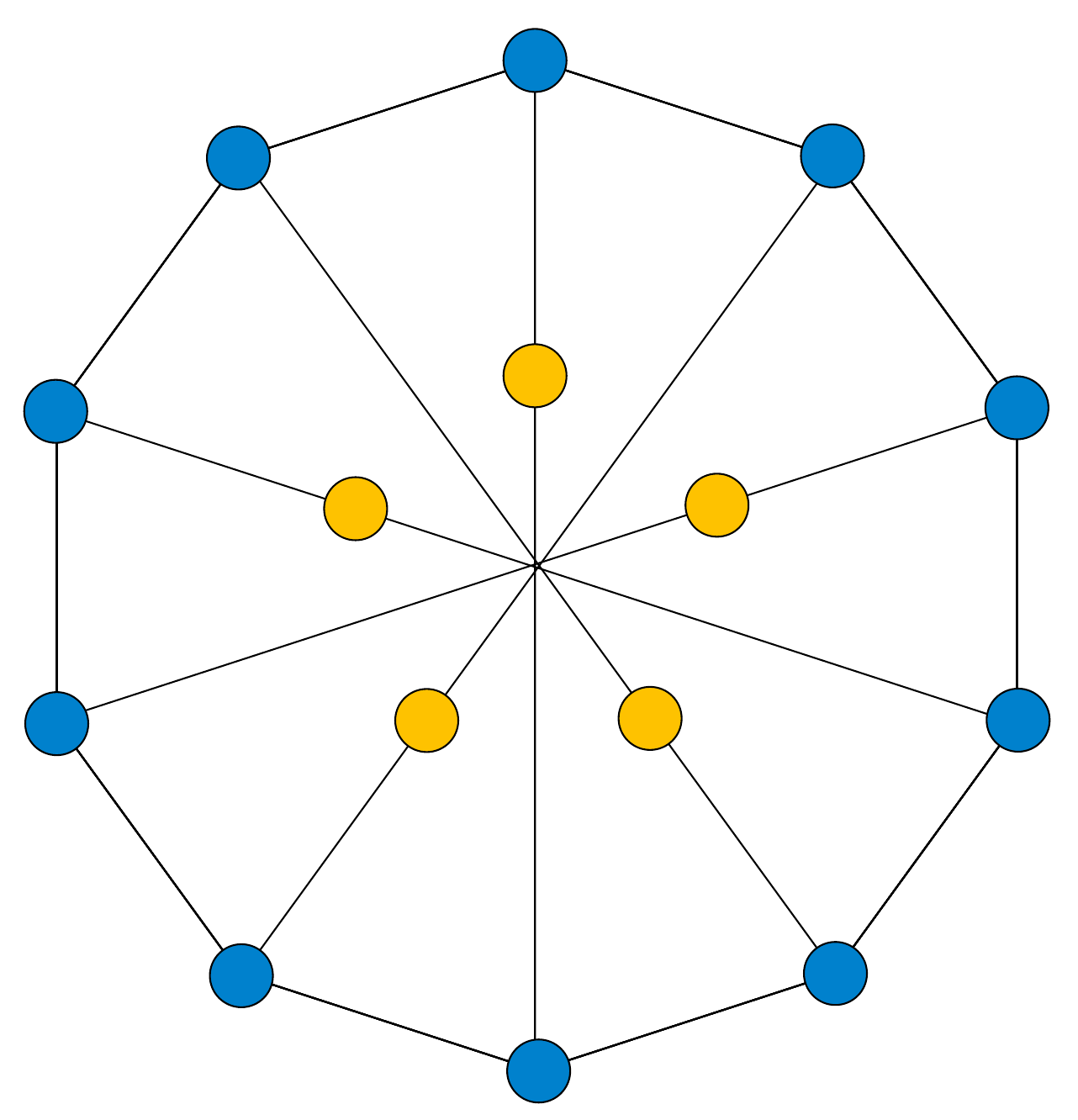}
    \caption{Subcubic obstacles $KT_3$ (left) and $SM_{10}$ (right).}
    \label{fig:subcubic-obstacles}
\end{figure}

\begin{proposition} \label{prop:KT3}
    $KT_{3}$ is a minimal non-string graph.
\end{proposition}

We note that additional obstructions can be obtained in a similar fashion by replacing multiple vertices of the $1$-subdivision of $K_{3,3}$ with triangles (rather than just one).

Our second subcubic obstacle, $SM_{10}$, is not obtained from $K_{3,3}$, and its existence demonstrates that the 1-subdivision of $K_{3,3}$ is not the only triangle-free subcubic obstacle. It may be formed from the 10-vertex Möbius ladder~\cite{GuyHar-CMB-67} by subdividing the five ``rungs'' of the ladder; it consists of a 10-cycle of degree-3 vertices, together with five two-edge paths connecting opposite pairs of vertices of the 10-cycle.

\begin{proposition} \label{prop:SM10}
$SM_{10}$ is a minimal non-string graph.
\end{proposition}

The proofs of \cref{prop:KT3} and \cref{prop:SM10} are direct applications of \cref{thm:subcubic-string}; for the full proofs see \cref{sec:subcubic-proofs}. For a partial proof that the 14-vertex Heawood graph is also a minimal non-string graph, see \cref{sec:d3-antichain}.

\subsection{Girth of Subcubic Obstacles}
Since $KT_3$ contains a triangle, there is no nontrivial lower bound on the girth of subcubic obstacles. On the other hand, in what follows we show an upper bound. We utilize the following well-known bound on the edge density of planar graphs with high girth:

\begin{lemma}
\label{lem:euler-bound}
    Let $G$ be a connected planar graph with at least $3$ vertices and with girth at least $g$. Put $n = |V(G)|$ and $m = |E(G)|$. Then
    \[ \frac{m}{n} \leq \frac{g}{g-2}. \]
\end{lemma}

We show that the properties of being a subcubic obstacle imply that the edge density of any such obstacle is bounded above $c$ for a constant $c > 1$, then use \cref{lem:euler-bound} to obtain an upper bound on the edge density of the obstacle that is smaller than $c$ if its girth is too large. The details are presented below.

\begin{theorem}
    Let $G$ be a subcubic non-string graph such that every proper induced minor of $G$ is a string graph. Then $G$ has girth smaller than $30$.
\end{theorem}

\begin{proof}
    We begin by noting several properties of $G$ that follow from the fact that $G$ is a subcubic obstacle. Firstly, $G$ is connected, and $G \setminus v$ is connected for any $v \in V(G)$ of degree $2$. Additionally, $G$ does not contain any vertex $v$ with $\deg(v) \in \{0,1\}$. Further, $G$ does not contain two adjacent vertices $u$, $v$ of degree two.

    Since $G$ has no vertices of degree $0$ or $1$, $G$ is not a forest, and thus has finite girth. Suppose for contradiction that $G$ has girth at least $30$; then $n := |V(G)| \geq 30$. Since every $e \in E(G)$ has at least one endpoint of degree $3$, $G$ has density at least $\frac{6}{5}$, which is the density of the graph obtained from a cubic graph by subdividing each edge once. If $G$ is cubic, let $v \in V(G)$ be any vertex such that $G \setminus v$ is connected, and otherwise let $v \in V(G)$ be a vertex of degree $2$. If $G$ is cubic then it is easily verified that $G \setminus v$ has density at least $\frac{6}{5}$. If $G$ is not cubic, then $G \setminus v$ has $n-1$ vertices and at least $\frac{6}{5}n - 2$ edges. Thus the density of $G \setminus v$ is at least $\frac{\frac{6}{5}(30) - 2}{30 - 1} = \frac{34}{29}$.

    Since $G \setminus v$ is a proper induced minor of $G$, it is a string graph, and so by \cref{thm:subcubic-string} there is a matching $M$ of $G \setminus v$ such that $(G \setminus v) / M$ is planar. Since $G \setminus v$ has girth bigger than $3$, contracting $M$ does not create any parallel edges, and so $(G \setminus v) / M$ has $n - |M|$ vertices and at least $\frac{34}{29}n - |M|$ edges, thus the density of $G$ is at least $\frac{\frac{34}{29}n - |M|}{n - |M|} \geq \frac{34}{29}$. $(G \setminus v)/M$ has girth at least $15$, since $G \setminus v$ has girth at least $30$ and at most half of the edges in any cycle of $G \setminus v$ can be contained in $M$. Since $G \setminus v$ is connected, $(G \setminus v)/M$ is also connected. As $(G \setminus v)$ is planar, by \cref{lem:euler-bound} we have that the density of $(G \setminus v)/M$ is at most $\frac{15}{13}$, which is smaller than $\frac{34}{29}$ and thus a contradiction.
\end{proof}

\section{Subcubic Recognition} \label{sec:recognition}

Courcelle's theorem~\cite{Courcelle-Book} is a powerful algorithmic metatheorem for converting logical statements about graphs into algorithms.
It concerns the \emph{monadic second-order logic of graphs} $\mathrm{MSO}_2$, which allows quantified formulas over vertices, edges, vertex sets, and edge sets, with predicates for testing set membership and vertex--edge incidence. According to Courcelle's theorem, the problem of testing whether a graph models a formula of this type can be solved in time that is linear in the number of vertices in the graph, multiplied by a (non-elementary but computable) function of the formula length and the treewidth of the graph. Thus, for any fixed formula, recognizing the graphs that model the formula is fixed-parameter tractable in the treewidth of the graph. In order to design an algorithm of this type for a graph recognition problem, it is necessary only to find a logical formula characterizing the graphs to be recognized.
This methodology has been frequently used for problems in graph drawing and topological graph theory~\cite{BanEpp-JGAA-18,ChaDijKry-JGAA-20,ChaYen-WADS-15,EppKinKob-Algo-18,MueRut-GD-24}, including by Courcelle himself~\cite{Cou-TCS-00a,Cou-TCS-00b}.

\begin{observation}
The characterization of subcubic string graphs of \cref{thm:subcubic-string} can be expressed as a formula in $\mathrm{MSO}_2$.
\end{observation}

\begin{proof}
At a high level, we express in logic the existence of a set of edges $M$ such that contracting $M$ produces a graph that has neither $K_5$ nor $K_{3,3}$ as a minor. The existence of a $K_5$ minor, or a $K_{3,3}$ minor, in the contracted graph, can be expressed in $\mathrm{MSO}_2$ logic as the existence of five or six vertex sets of the original uncontracted graph, corresponding to the vertices of the minor, with the following properties:
\begin{itemize}
\item Any two of these vertex sets are disjoint: there does not exist a vertex belonging to any two of them.
\item Each of these five or six vertex sets forms a connected subgraph of the given graph. That is, for  each such set $S$, there is no cut $(S\cap T,S\setminus T)$ of the subgraph induced by $S$. Here, the existence of a cut may be expressed logically as the existence of a vertex set $T$ for which $S\cap T$ and $S\setminus T$ are nonempty and for which there does not exist an edge with one endpoint in $S\cap T$ and the other endpoint in $S\setminus T$.
\item For each edge $(u,v)$ in the minor, represented by the vertex sets $U$ and $V$, there exists an edge in the given graph whose two endpoints belong to $U$ and $V$.
\item The minor respects the contraction specified by $M$: for each edge $e$ in $M$, it is not the case that one of the endpoints of $e$ belongs to one of the five or six vertex sets of the minor but the other endpoint does not belong to the same set.
\end{itemize}
Each of these properties is straightforward to formalize in $\mathrm{MSO}_2$ logic.
\end{proof}

\begin{corollary}
Subcubic string graphs can be recognized in time that is fixed-parameter tractable with respect to treewidth. Within any subcubic graph family $\mathcal{F}$ of bounded treewidth, string graphs can be recognized in time that is linear in the representation of the input graph.
\end{corollary}

\bibliographystyle{plainurl}
\bibliography{string}

\newpage
\appendix

\section{Full proofs of subcubic obstacles} \label{sec:subcubic-proofs}

\begin{proof}[Proof of \cref{prop:KT3}]
    We first show that $KT_3$ is not a string graph via \cref{thm:subcubic-string}. We refer to vertices of $KT_3$ according to their colors in \cref{fig:subcubic-obstacles}. The only choices of cubic matchings to contract are: one of the red--red edges, one of the red--blue edges, both of the red--blue edges, or one red--blue edge and the opposite red--red edge. For each of these choices, contracting the matching produces a graph that contains (as a subgraph) a subdivision of $K_{3,3}$, and so is non-planar. It follows that $KT_3$ is a non-string graph.

It remains to show that every proper induced minor of $KT_3$ is a string graph. Any vertex deletion results directly in a planar graph. Let $f$ be an edge of $KT_3$; we show that there is a cubic matching $M$ of $KT_3 / f$ such that $(KT_3/f)/M$ is planar. For convenience, in what follows we give the vertex arising from the contraction of $f$ the color white, and the other vertices retain their colors. We have the following cases:
\begin{itemize}
    \item If $f$ is a red--red edge, then contracting either one of the two white--blue edges planarizes the graph.
    \item If $f$ is a red--blue edge, then contracting either one of the two white--red edges planarizes the graph.
    \item If $f$ is the red--yellow edge, then contracting the white--blue edge and the two red--blue edges planarizes the graph.
    \item If $f$ is a blue--yellow edge, then the white vertex has three neighbors. If any of its neighbors is blue, then contracting the white--blue edge planarizes the graph. Otherwise,  the white vertex has a red neighbor, and then contracting the white--red edge and the two red--blue edges planarizes the graph.
\end{itemize}
    By \cref{thm:subcubic-string}, in each case $KT_3/f$ is a string graph, from which the result follows.
\end{proof}

\begin{proof}[Proof of \cref{prop:SM10}]
To prove that $SM_{10}$ is a non-string graph, by \cref{thm:subcubic-string}, it suffices to show that every cubic matching is not planarizing. Any such matching contracts some of the edges of the cycle of ten degree-3 vertices in $SM_{10}$, leaving it still in the form of a shorter cycle. Following the Auslander–Parter planarity testing algorithm~\cite{AusPar-JMM-61,DiBEadTam-GD-98}, for the result to be planar, the \emph{pieces} of the graph with respect to this cycle (the two-edge paths through degree-2 vertices) must have an \emph{interlacing graph} that is bipartite. Here, two pieces are adjacent in the interlacing graph when the vertices where they attach to cycle are in alternating order around the cycle. In $SM_{10}$ itself, the interlacing graph is isomorphic to $K_5$, with each piece interlacing with two pieces that are consecutive to it around the 10-cycle and two more pieces that are non-consecutive. The non-consecutive interlacing pairs form a 5-cycle that is unaffected by contracting a matching, so the interlacing graph is non-bipartite and the contracted graph is non-planar.

To prove that $SM_{10}$ is minimal as a non-string graph, we consider cases:
\begin{itemize}
\item If a degree-3 vertex is deleted from $SM_{10}$, the result is planar.
\item If an edge incident to a degree-2 vertex is contracted, the result is a subcubic graph that can be planarized by contracting the other edge incident to the same degree-2 vertex. Therefore, the result of the first contraction is a string graph by \cref{thm:subcubic-string}.
\item If a cubic edge is contracted, the resulting induced minor is not cubic, but it nevertheless has a planarizing matching depicted in \cref{fig:SM10-contraction}  (left). Therefore, it is a string graph.
\item If a degree-2 vertex is deleted, the resulting induced minor has a planarizing matching depicted in \cref{fig:SM10-contraction}  (right). Therefore, it is a string graph.
\end{itemize}
\end{proof}

\begin{figure}
    \centering
    \includegraphics[width=\linewidth]{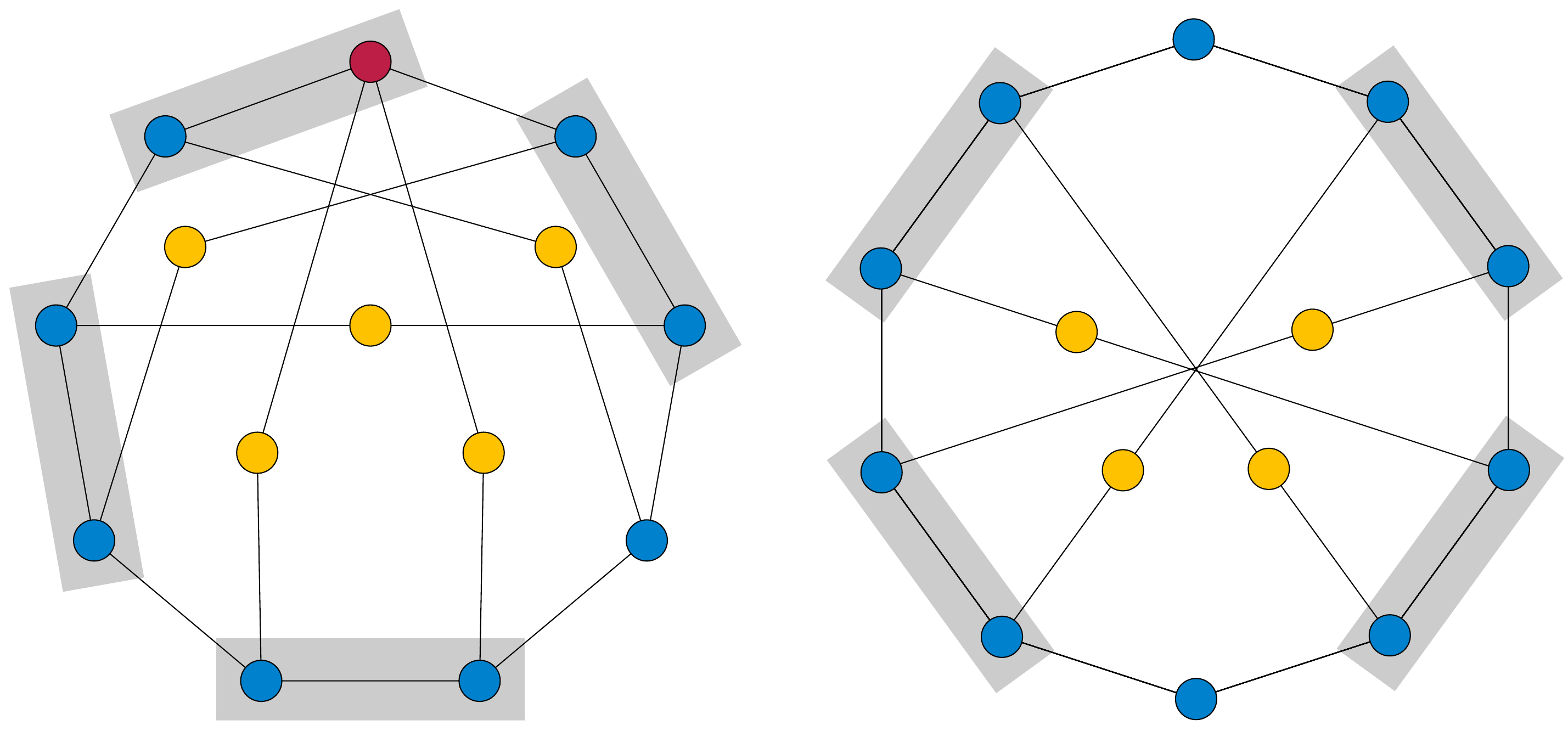}
    \caption{Left: a planarizing matching for the induced minor obtained by contracting a cubic edge of $SM_{10}$. Right: a planarizing matching for the induced minor obtained by deleting a degree-2 vertex of $SM_{10}$.}
    \label{fig:SM10-contraction}
\end{figure}

\section{A cubic antichain}
\label{sec:d3-antichain}

\begin{figure}[t]
\includegraphics[width=\textwidth]{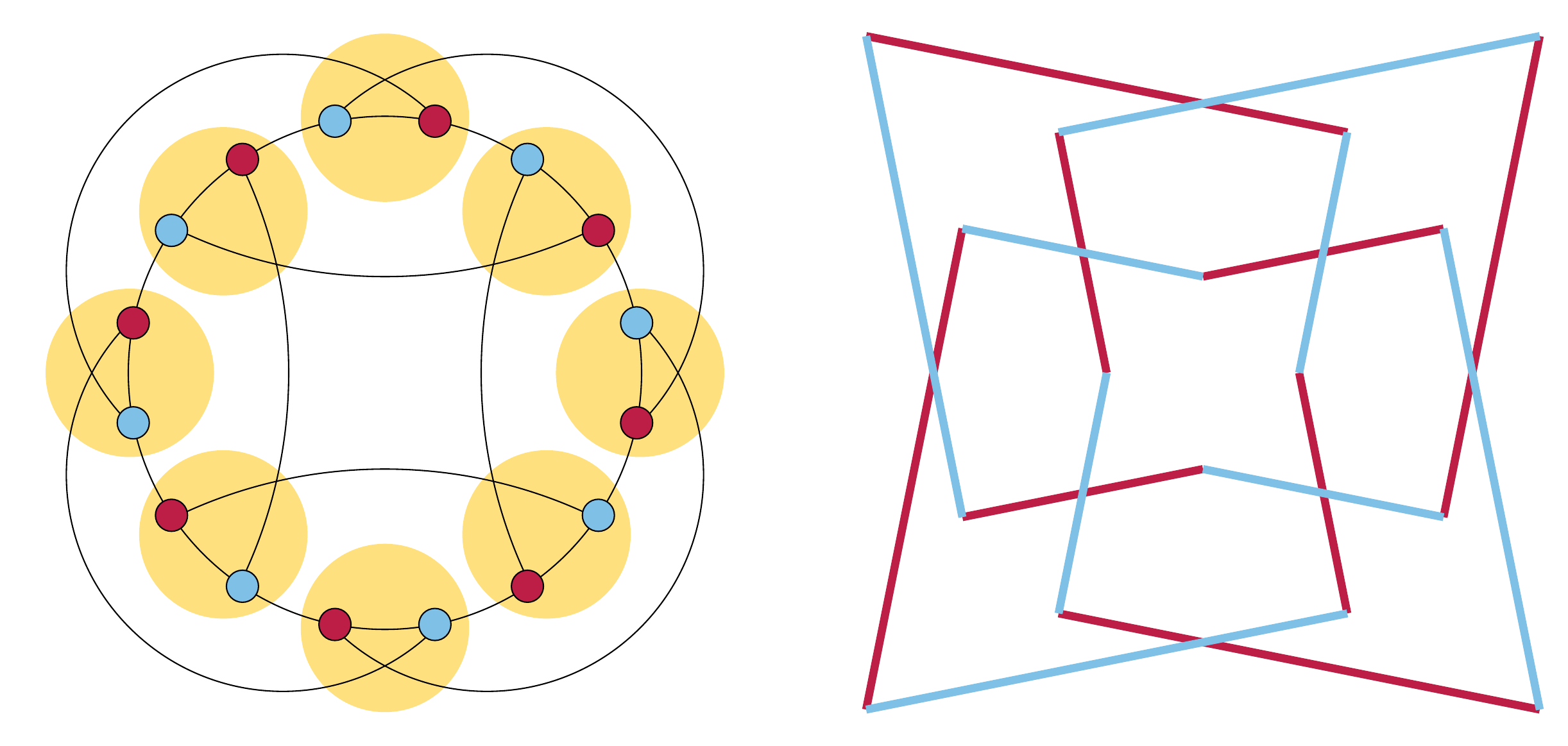}
\caption{The Möbius–Kantor graph. Left: the vertices are arranged in order around a Hamiltonian cycle, with chords alternating five steps forward and back, according to its LCF notation $[5,-5]^8$. The non-path edges highlighted in yellow form a planarizing matching. Right: the corresponding string representation. The cycle of end-to-end contacts of strings is a different Hamiltonian cycle consisting of the path edges.}
\label{fig:mk-string}
\end{figure}

Matoušek, Nešetřil, and Thomas~\cite{MatNesTho-CMUC-88} describe an infinite antichain in the induced minor order, consisting of planar graphs of maximum degree eight. None of these graphs can be transformed into another graph in the family by taking induced minors. 

In this section we describe an infinite family of cubic graphs which form an antichain in the induced minor order. These graphs are described in LCF notation~\cite{Fru-JGT-76} by the notation $[5,-5]^n$; this means that they consist of a $2n$-vertex Hamiltonian cycle together with chords that alternate between advancing five steps forward along the cycle and five steps backward. Several well-known graphs have this form: the Heawood graph or 6-cage is $[5,-5]^7$, the Möbius–Kantor graph is $[5,-5]^8$ (\cref{fig:mk-string}), and the smallest zero-symmetric graph is $[5,-5]^9$. We will assume $n\ge 7$ to ensure that the girth of the resulting graphs equals six.

\begin{theorem}
The graphs $[5,-5]^n$ with $n\ge 7$ form an antichain in the induced minor ordering.
\end{theorem}

\begin{proof}
First, observe that in the neighborhood of any vertex $v$ of one of these graphs, we can decompose the graph into three vertex-disjoint paths using edges that advance in alternating $+1$ and $+5$ steps along the Hamiltonian cycle. (Globally, these paths might connect into a single Hamiltonian cycle, as they do in \cref{fig:mk-string}, or three disjoint cycles, depending on $n$.) Thus, within this neighborhood, the part of the neighborhood ahead of $v$ around the Hamiltonian cycle and the part of the neighborhood behind $v$ are connected by three vertex-disjoint paths. Call the edges of these paths \emph{path edges} and the remaining edges \emph{non-path edges}; the non-path edges are the ones between a forward and backward five-step chord of the Hamiltonian cycle, highlighted in the yellow matching in \cref{fig:mk-string}. If any vertex of the graph were removed in the process of forming an induced minor, or if any non-path edge connecting vertices from two different local paths were contracted, the local connectivity would be reduced to two, and could not be restored by additional deletions or contractions. Therefore, no such operation can be used in forming a graph $[5,-5]^m$ as an induced minor of another graph $[5,-5]^n$.

\begin{figure}[t]
\includegraphics[width=\textwidth]{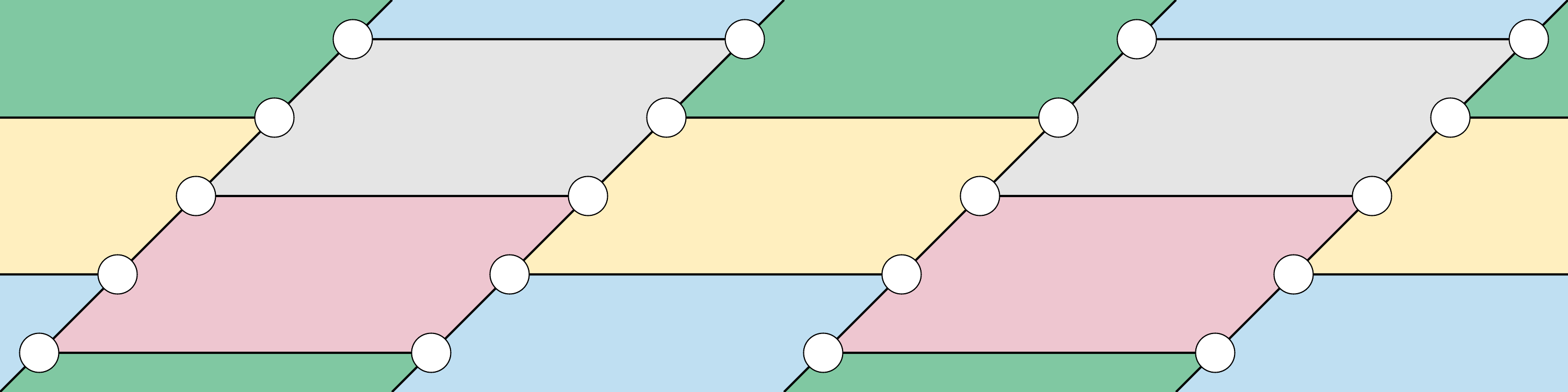}
\caption{Torus embedding of $[5,-5]^{10}$.}
\label{fig:torus-20}
\end{figure}

It remains to show that contractions of path edges cannot produce $[5,-5]^m$ from $[5,-5]^n$. Along with the obvious 6-cycles formed by a chord and five edges of the Hamiltonian cycle, $[5,-5]^n$ has many additional 6-cycles whose edges, in terms of the difference in positions of their endpoints along the Hamiltonian cycle, follow the pattern $+1,+1,+5,-1,-1,-5$. These additional 6-cycles form the faces of a torus embedding of $[5,-5]^n$ (\cref{fig:torus-20}), so we call them \emph{face cycles}. If any edge of one of these face cycles is contracted, this contraction will reduce the girth of the graph, unless at least four edges of the same face are all contracted. Because every edge of $[5,-5]^n$ belongs to a face cycle, and each face cycle has exactly four path edges, the only way to avoid reducing the girth of the graph, using only contractions of uncrossed edges, is to contract all four path edges of at least one face cycle. But this also contracts an edge of the face cycle two steps ahead in the Hamiltonian cycle, which must again have all four of its path edges contracted, and so on. So any contraction of path edges that avoids reducing the girth of the graph must contract all path edges in the entire graph, which cannot result in another graph of the form $[5,-5]^m$.
\end{proof}

\begin{figure}[t]
\includegraphics[width=\textwidth]{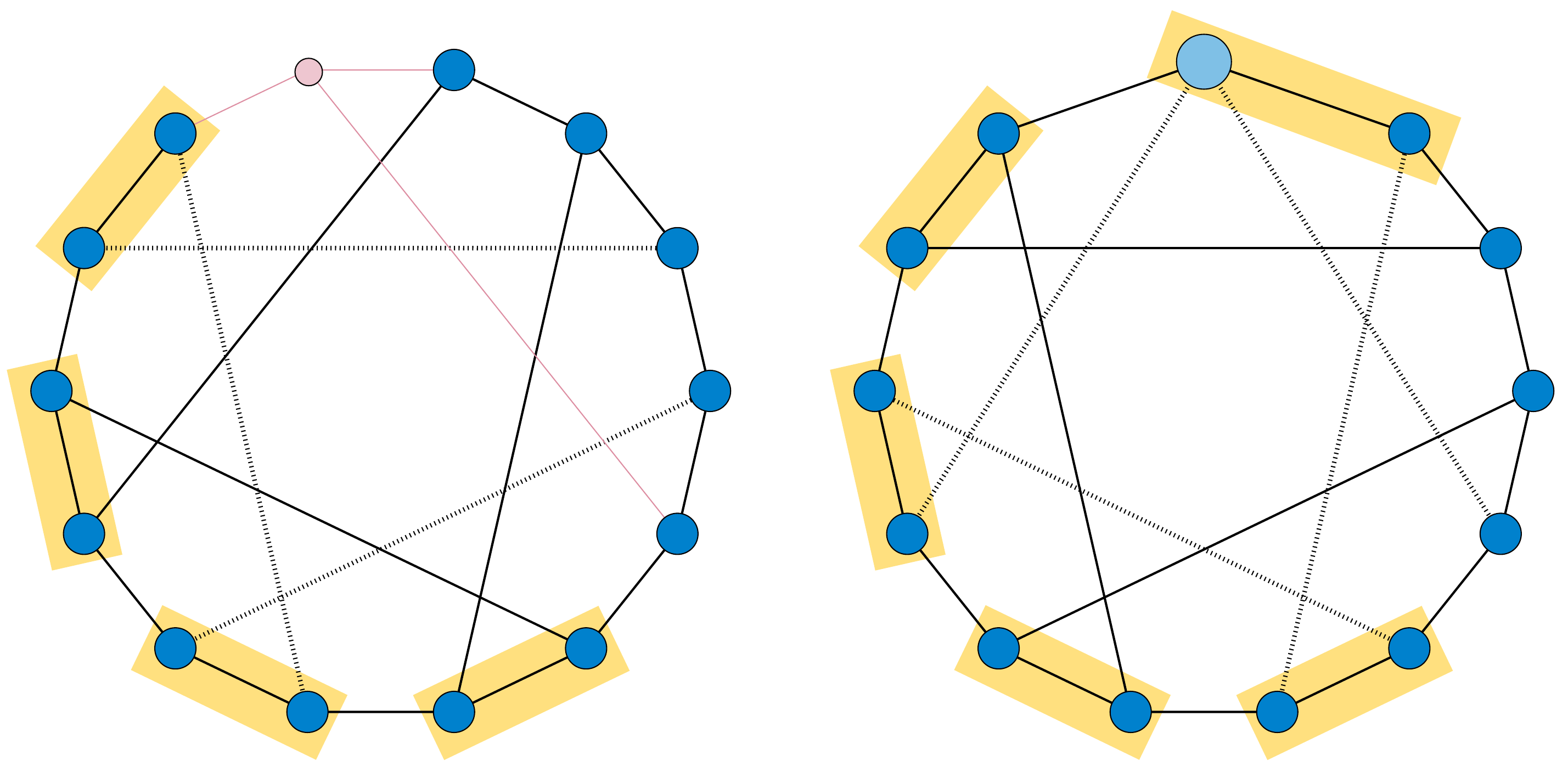}
\caption{Planarizing matchings for proper minors of the Heawood graph. Left: Deletion of the light red vertex. Right: Contraction of an edge into the light blue degree-4 vertex.}
\label{fig:heawood-minimality}
\end{figure}

In an unpublished preprint, Moshe White proves that the Heawood graph $[5,-5]^7$ is not a string graph~\cite{Whi-23}. Assuming this proof to be valid, this graph provides another new cubic obstacle to string graphs:

\begin{observation}
Every proper induced minor of the Heawood graph is a string graph.
\end{observation}

\begin{proof}
Because the Heawood graph is both vertex-transitive and edge-transitive, we need consider only two proper induced minors, obtained by deleting one vertex or by contracting one edge. Planarizing matchings are shown for each in \cref{fig:heawood-minimality}, with matched pairs of vertices indicated by the yellow shading. The solid and dashed coloring of the chords of the remaining Hamiltonian cycle indicates a bipartition of the interlace graph with respect to this cycle, indicating that the result is planar. It can be drawn planarly (after contracting the matched pairs) by rerouting the dashed edges outside the Hamiltonian cycle.
\end{proof}

As \cref{fig:mk-string} shows, the Möbius–Kantor graph $[5,-5]^8$ is a string graph, and this can be generalized to $[5,-5]^n$ for even $n$:

\begin{observation}
Every graph $[5,-5]^n$ for even $n$ is a string graph.
\end{observation}

\begin{proof}
It has a planarizing matching consisting of all non-path edges, such as the one in \cref{fig:mk-string}. Contracting this matching produces a planar graph isomorphic to an $(n/2)$-antiprism.
\end{proof}

We do not know whether $[5,-5]^n$ is a string graph for other odd values of $n$, or if so whether a subcubic obstacle can be obtained from it.

\end{document}